\documentclass[10pt]{article}
\font\smallit=cmti10
\font\smalltt=cmtt10

\usepackage{amssymb,latexsym,amsmath,epsfig,amsthm,comment,tasks} 
\usepackage[table]{xcolor}
\usepackage{hyperref}

\usepackage{seqsplit}
\usepackage{listings}

\settasks{
  label=\roman*),
  label-width=2em,
  label-align=right,
  label-offset=0.5em
}

\makeatletter

\renewcommand\section{\@startsection {section}{1}{\z@}
{-30pt \@plus -1ex \@minus -.2ex}
{2.3ex \@plus.2ex}
{\normalfont\normalsize\bfseries\boldmath}}

\renewcommand\subsection{\@startsection{subsection}{2}{\z@}
{-3.25ex\@plus -1ex \@minus -.2ex}
{1.5ex \@plus .2ex}
{\normalfont\normalsize\bfseries\boldmath}}

\newcommand\blfootnote[1]{
\begingroup
\renewcommand\thefootnote{}\footnote{#1}
\addtocounter{footnote}{-1}
\endgroup
}

\newcommand{\bb}[1]{\mathbb{#1}}

\makeatother

\newtheorem{theorem}{Theorem}
\newtheorem{lemma}[theorem]{Lemma}
\newtheorem{proposition}[theorem]{Proposition}
\newtheorem{corollary}[theorem]{Corollary}

\theoremstyle{definition}

\def\cH{{\mathcal H}}

\def\F{{\mathbb F}}

\def\Q{{\mathbb Q}}

\def\Z{{\mathbb Z}}

\def\\{\cr}
\def\({\left(}
\def\){\right)}
\def\[{\left[}
\def\]{\right]}
\def\<{\langle}
\def\>{\rangle}

\newcommand{\ord}{\operatorname{ord}}
\newcommand{\Char}{\operatorname{char}}

\def\Hb1{\overline{\cH}_{m}}
\def\Ht1{\widetilde{\cH}_{m}}

\newcolumntype{L}{>{\raggedright\arraybackslash}X}

\begin{document}
\begin{figure}
\vspace*{-62pt}
\vspace{.75in}
\end{figure}
\vspace*{-45pt}
\leftline{\smalltt } \vskip -12.5pt
\centerline{\smalltt }
\vskip 3pt \hrule																													
\begin{center}
\uppercase{\bf \boldmath Primitive normal progressions in finite fields}
\vskip 20pt
{\bf Gustav Kj\ae rbye Bagger}\\
{\smallit School of Science, UNSW Canberra, Australia}\\ 
{\tt g.bagger@unsw.edu.au}\\ 
\vskip 20pt
{\bf Matthew Fernando}\\
{\smallit School of Science, UNSW Canberra, Australia}\\ 
{\tt p.maseuge\_fernando@unsw.edu.au}\\ 
\end{center}
\blfootnote{\emph{Date:} \today}
\blfootnote{\emph{Corresponding author:} Matthew Fernando (p.maseuge\_fernando@unsw.edu.au)}
\blfootnote{\emph{Key phrases:} Primitive elements, normal elements, prime sieves}
\vskip 15pt
\centerline{\bf Abstract}
\noindent
Consider an arithmetic progression inside the finite field $\F_{q^n}$. We study the existence of progressions whose terms are all primitive and at least one of them is normal. Using a combination of sieve estimates, ramified bounds, and computational arguments, we improve significantly upon previous literature for progressions of length $m=4$ and obtain partial results whenever $5\leq m\leq 10$. Finally, we provide a general bound which outperforms previous results in the range $m\geq 53$.

\pagestyle{myheadings}
\markright{\smalltt Arithmetic progressions in finite fields \hfill}
\thispagestyle{empty}
\baselineskip=12.875pt
\vskip 30pt

\section{Introduction}
Let $\F_{q^n}$ be the finite field with $q^n$ elements for some prime power $q$. An element $\gamma \in \F_{q^n}$ is called \emph{primitive} if it generates $\F_{q^n}^*$, a cyclic group of order $q^n-1$. We say that $\Gamma \in \F_{q^n}$ is \emph{normal} if $\{ \Gamma, \Gamma^q, \ldots, \Gamma^{q^{n-1} }\}$ is an $\F_q$ basis for $\F_{q^n}$ as a vector space. Consider an arithmetic progression of length $m \geq 3$ inside $\F_{q^n}$, given by 
$$P(\alpha,\beta): \quad \alpha,\alpha+\beta,\alpha + 2\beta,\dots,\alpha + (m-1)\beta, $$
for some $\alpha \in \F_{q^n}$ and some non-zero $\beta \in \F_{q^n}^*$. If every element in $P(\alpha,\beta)$ is primitive and at least one element is normal, call the progression \emph{m-primitive}.
We define the set $N_m$ consisting of pairs $(q,n)$ satisfying the following:
$$\text{For any } \beta \in \F_{q^n}^*, \text{ there exists } \alpha \in \F_{q^n} \text{ such that } P(\alpha,\beta) \text{ is $m$-primitive}. $$
Note that, for any $m\geq 4$, we may discard either the first or last element in an $m$-primitive progression $P(\alpha,\beta)$ to obtain an $(m-1)$-primitive progression. Moreover, for any $m > \Char \F_q$, the sequences $P(\alpha,\beta)$ contain duplicate entries. Thus, such duplicates may be removed and $N_m = N_{\Char (\F_q)}$ for any $m\geq \Char \F_q$. We have a tower of inclusions
$$ N_3 \supseteq \dots \supseteq N_{\Char(\F_q)-1} \supseteq N_{\Char(\F_q)} = \dots $$
Throughout this paper, we assume that $m\leq \Char \F_q.$
Lemos, Neumann and Ribas completely determined the set $N_m$ for $m = 3$, see~\cite[Proposition 4.3]{aguirre-lemos-neumann-ribas-2022} for more details. In the case where $m=4$, they proved the following result.
\begin{proposition}[{\cite[Theorem 1.1 (iii)]{lemos-neumann-ribas-2023}}] \label{P: 4.6 in Lemos}
Let $q$ be a prime power with $\Char\F_{q^n} \geq 5$, and $n \geq 2$ be an integer. Then
$$q^n \geq 3.31 \times10^{2821} \quad \text{implies} \quad (q,n) \in N_4. $$
\end{proposition}
This follows from estimating character sums and sieve methods. We prove an improved classification result for $N_4$ via a prime sieve argument together with a ramified lower bound on $q^n-1$ as in~\cite{bagger2026hybrid}.
\begin{theorem}\label{T: main classification result}
Let $q$ be a prime power with $\Char\F_{q^n} \geq 5$, and $n \geq 2$ be an integer. Let $E_4$ be the set of $714$ possible exceptional pairs $(q,n)$ listed in the appendix, the largest such $q$ being $326041.$ Then
$$ (q,n) \notin \left\{(q,2): q \leq 3.82 \times 10^{15} \right\} \cup E_4 \quad \text{implies} \quad (q,n) \in N_4. $$
\end{theorem}
In the range $m \in [5,10]$, we also obtain an improvement.
\begin{theorem}\label{T: medium m}
Let $q$ be a prime power with $\Char\F_{q} \geq m$. Then $(q,n) \in N_m$, provided
\begin{tasks}(2)
    \task $m=5$ \ and $q^n\geq 8.6 \times 10^{171}$
    \task $m=6$ \ and $q^n\geq 1.3 \times 10^{384}$
    \task $m=7$ \ and $q^n\geq 6.8 \times 10^{895}$
    \task $m=8$ \ and $q^n \geq 1.5 \times 10^{2121}$
    \task $m=9$ \ and $q^n \geq 8.5 \times 10^{5086}$
    \task $m=10$ and $q^n \geq 3.9 \times 10^{12228}$.
\end{tasks}
\end{theorem}
For longer progressions where $m\geq 11$, we instead prove a general result.
\begin{theorem}\label{General m theorem (m>=11)}
Let $m \geq 11$ be an integer, and define
\begin{align*}
A(m)&=\frac{\log\log(m^2)} {\log\log(m^2-1)} & \text{ and } & &B(m)=\frac{1.919A(m)}{\left(0.999-\frac{\log 4}{\log m}\right)}.
\end{align*}
Then, for every odd prime power $q$ with $\Char \F_q \geq m$ and for every $n\geq2$,
$$ q^n\geq \exp\left(\exp(m\cdot B(m))\right)
\quad\text{ implies }\quad
(q,n)\in N_m.$$
\end{theorem}
This result out-performs the asymptotic found in~\cite{lemos-neumann-ribas-2023} whenever $m\geq 53$.

The outline of this paper is as follows. In Section 2, we develop some elementary results. In Section 3, we refine the bound on $q^n$ by considering various choices of polynomials $g\in \Z[x]$ and partition the range in which exceptions can lie by the extension degree $n$. Whenever $q\leq 10^4$, we resolve $N_4$-membership by direct computation. In Section 4, we handle the cases where $n=3,4$ and $5$ through a mix of branching techniques and the ramified bound, culminating in the proof of Theorem \ref{T: main classification result}. In Section 5, we discuss applications to larger values of $m$, proving Theorems~\ref{T: medium m} and~\ref{General m theorem (m>=11)}. All computations in this article were carried out using SageMath~\cite{sagemath} and Python.

\section{Preliminaries}
In order to state our main sieving criterion for $N_m$ membership, we require a weakening of the notion of primitiveness. For each divisor $e\mid q^n-1$, call an element $\gamma \in \F_{q^n}^*$ \emph{e-free} if there is no $\beta \in \F_{q^n}^*$ satisfying $\gamma = \beta^d$, for any $d>1$ dividing $e$. It follows that $\gamma \in \F_{q^n}^*$ is primitive if and only if it is $(q^n-1)$-free. Similarly, we may define the notion of $g$-freeness for $g\in \F_q[x]$, a divisor of $x^n-1$. Call $\alpha \in \F_{q^n}$ \emph{g-free} if, for any divisor $h \mid g$ with $h = \sum_{i=0}^ra_ix^i \in \F_q[x] \setminus \{1\}$, we have
$$ h \circ \beta := \sum_{i=0}^r a_i \beta^{q^i} \neq \alpha \quad \text{for any } \beta \in \F_{q^n}.$$
Note that the action 
$$\circ:\F_q[x] \times \F_{q^n} \to \F_{q^n} $$ 
endows $\F_{q^n}$ with a $\F_q[x]$-module structure. Throughout, we let $W(x)$ be the square-free divisor counting function, $\omega(x)$ be the prime divisor counting function and $\varphi(x)$ be Euler's totient function. We will regularly use the shorthand $\omega := \omega(q^n-1)$, whenever the choice of field $\F_{q^n}$ is clear. Let $\ord_a(x)$ denote the multiplicative order of $x$ in $(\Z/a\Z)^*$, which is well-defined whenever $(a,x) = 1$. For a given polynomial $g\in \F_q[x]$, define $W(g)$ to be the amount of square-free monic divisors of $g$. Throughout, we will assume that $\Char \F_q \geq 5$. We are now in a position to state the main sieving criterion for $N_m$ membership.

\begin{proposition}\label{Sieving result}
Let $q$ be a prime power, $n\geq 2$ and $m\geq 3$, all integers. Consider divisors $e \mid q^n-1$, $g \mid x^n-1$ and take $\{p_1,p_2, \ldots, p_r \} \subset \Z$, $\{h_1,h_2,\ldots, h_s\} \subset \F_q[x]$ to be the set of prime factors of $q^n-1$, coprime to $e$, and the set of monic irreducible factors of $x^n-1$, coprime to $g$, respectively. Suppose
$$\delta:= 1-\sum_{i=1}^{s}\frac{1}{q^{\deg h_i}}-m\sum_{i=1}^{r}\frac{1}{p_i}>0. $$
Then
\begin{equation} \label{Sieve criterion}
 q^{\frac{n}{2}}\geq m \left(2 + \frac{mr+s-1}{\delta}\right) W(g)W(e)^m \quad \text{implies} \quad (q,n) \in N_m.  
\end{equation}

\end{proposition}
A proof of this fact can be found in~\cite[Proposition 3.8]{lemos-neumann-ribas-2023}. The crucial idea is to consider the counts of sequences $P(\alpha,\beta)$ where all elements are $e$-free for some $e \mid q^n-1$ and at least one element is $g$-free for some $g \mid x^n-1$. By considering the indicator functions for such elements and using a modified inclusion-exclusion argument, one is able to `glue' together these counts into a count of $(q^n-1)$-free sequences with at least one $(x^n-1)$-free element. For a more detailed outline of this method, see the proof of Theorem 1 in~\cite{BaggerPunch2025}. Finally, one can show that an element $\alpha \in \F_{q^n}$ is $(x^n-1)$-free if and only if it is normal, completing the argument.

Since we will often use Proposition~\ref{Sieving result} when $q^n-1$ is large enough to make factoring infeasible, we will apply a slightly modified version of the criterion. This version will depend only on the quantities $n$ and $\omega(q^n-1)$. Indeed, the only dependence on the specific prime factorisation of $q^n-1$ in Proposition~\ref{Sieving result} is in the definition of $\delta$ and the amount of prime factors chosen, denoted by $r$. In addition, we may choose these $r$ factors freely in order to minimize $\delta$. The sum of reciprocals $\sum_{i=1}^rp_i^{-1}$ is bounded above by $\sum_{i=\omega-r+1}^\omega \hat{p}_i^{-1}$, where $\hat{p}_i$ denotes the $i^\text{th}$ rational prime. Moreover, since $\F_q[x]$ is an integral domain, any $g \in \F_q[x]$ splits into at most $\deg g$ linear factors. Therefore $\sum_{i=1}^{s}q^{-\deg h_i}$ is bounded above by $(n - \deg g)q^{-1}$. If we define
\begin{equation*}
\delta_\omega^{'} := 1 - \sum_{i=1}^{s}\frac{1}{q^{\deg h_i}}- m\sum_{i=\omega-r+1}^\omega \frac{1}{\hat{p}_i} \  
\end{equation*}
\begin{equation*}
\delta_\omega := 1 - \frac{n-\deg g}{q} - m\sum_{i=\omega-r+1}^\omega \frac{1}{\hat{p}_i},    
\end{equation*}
we therefore have $\delta \geq \delta'_\omega \geq \delta_\omega$. 
In the cases where $q^n-1$ is large, we invoke Proposition~\ref{Sieving result} with either $\delta_\omega$ or $\delta'_\omega$ instead of $\delta$ to avoid factoring. 

In order to evaluate $W(g)$ for an arbitrary divisor $g \mid x^n-1$, we need a good handle on the monic irreducible factors of $x^n-1$ over $\F_q$. If we work over the rationals instead, we have the factorisation
$$ x^n-1 = \prod_{d\mid n}\Phi_d(x), $$
where $\Phi_d(x)$ is the $d^\text{th}$ cyclotomic polynomial, irreducible over $\Q$. Indeed $\Phi_d$ is monic with coefficients in $\Z$. However, over $\F_q$, these polynomials may be reducible depending on the splitting of $q$ in the cyclotomic extension $\Q(\zeta_d)/\Q$. Assume $q=p^k$ and $p\nmid d$. By cyclotomic reciprocity,
$$ \Phi_d(x) = \prod_{i =1 }^{\varphi(d)/\ord_d(q)} \phi_i(x),$$ 
where $\phi_i(x) \in \F_q[x]$ and $\deg \phi_i = \ord_d(q)$. In particular, if $p \nmid n$, then 
$$ \#\{\text{monic irreducibles of } x^n-1\} = \sum_{d\mid n} \frac{\varphi(d)}{\ord_d(q)}. $$
In the case where $d = p^a d'$ for some $a>0$ and $p\nmid d'$, the situation becomes more delicate since $\Phi_d(x) = \Phi_{d'}(x)^{p^a}$. If $\Char \F_q = p$ and $n = p^{v_p(n)}m$, we have
$$ x^n -1 = \prod_{d \mid m} \Phi_d(x)^{\frac{v_p(n)+1}{2}(1 + \frac{n}{m})}. $$
Thus, if $p\mid n$, we have
$$ \#\{\text{monic irreducibles of } x^n-1\} = \sum_{d\mid np^{-v_p(n)}} \frac{\varphi(d)}{\ord_d(q)}. $$
We avoid these intricacies by using an upper bound on the monic irreducible factors.
\begin{lemma}~\cite[Lemma 4.3]{lemos-neumann-ribas-2023}\label{L: 4.3 in Lemos}
    Let $q$ be a prime power and let $n$ be a positive integer. The number of monic irreducible factors of $x^n-1$ over $\mathbb{F}_q$ is at most $\frac{n}{a}+b,$ where the pair $(a,b)$ can be chosen to be one of the following pairs:
    \begin{align*}
        (1,0), \hspace{1.5cm} \left(2, \frac{q-1}{2}\right), &\hspace{1cm} \left( 3, \frac{q^2+3q-4}{6} \right),\\
        \left(4, \frac{q^3 + 3q^2+5q-9}{12} \right), & \hspace{1cm}  \left(5,\frac{3q^4+8q^3+15q^2+22q-48}{60} \right).
    \end{align*}
\end{lemma}
Following the proof of~\cite[Lemma 4.3]{lemos-neumann-ribas-2023}, one may obtain other pairs $(a,b)$ of the form $\left(k, \frac{q^{k-1}}{k(k-1)} + O(q^{k-2}) \right)$ for any integer $k \geq 6.$ We also require a bound on $W(q^n-1)$ in order to apply Proposition \ref{Sieving result} with $g=x^n-1$ and $e=q^n-1$. 

\begin{lemma}\cite[Lemma 4.1]{lemos-neumann-ribas-2023}\label{L: 4.1 in Lemos}
Let $r$ be a positive integer, $p_1, \ldots, p_r$ be the list of first $r$ prime numbers, and $P_r=p_1\cdots p_r$. For every positive integer $u \geq P_r,$ we have $W(u) \leq u^t$ where $t\in \mathbb{R}$ satisfies $t \geq \frac{r \log 2}{\log P_r}$.    
\end{lemma}
Proposition \ref{P: 4.6 in Lemos} was obtained in~\cite{lemos-neumann-ribas-2023} by applying Proposition \ref{Sieving result} with $g=x^n-1$ and $e=q^n-1$. The argument uses Lemma \ref{L: 4.3 in Lemos} and Lemma \ref{L: 4.1 in Lemos} to obtain upper bounds for $W(x^n-1)$ and $W(q^n-1),$ respectively. We omit the details of this argument here and return to these estimates in Section \ref{Section on general m}, where we consider the case of general $m$.

We are now in a position to improve on Proposition~\ref{P: 4.6 in Lemos}. 
\begin{proposition}\label{Reduction of region to 10^82}
    Let $q$ be an odd prime power, and $n\geq 2$ be an integer. Then
    $$ q^n \geq 1.50\times 10^{82} \quad \text{implies} \quad (q,n) \in N_4.$$
\end{proposition}
\begin{proof}
    We will use the same notation as in Proposition \ref{Sieving result}. Since the the $N_4$-membership is guaranteed whenever $q^n\geq 3.31 \times 10^{2821}$ by Proposition \ref{P: 4.6 in Lemos} we may assume $q^n < 3.31\times 10^{2821}$. Considering the product of the first 853 primes, we have
    $$\prod_{i=1}^{853}\hat{p}_i \geq 2.183 \times 10^{2825}.$$
    In particular, it is sufficient to consider pairs $(q,n)$ for which $\omega \leq 852$.
    We apply Proposition \ref{Sieving result} with the  fixed choice of $g=x^n-1$ and bounding $W(x^n-1)$ via Lemma~\ref{L: 4.3 in Lemos}. Given $\delta'_\omega>0$, we have the criterion
    $$ q^{n} \geq 16\left(2+\frac{4r-1}{\delta'_\omega}\right)^2 4^{4(\omega-r)+\frac{n}{a} + b} \quad \text{implies} \quad (q,n)\in N_4. $$
    If $c\leq q$, then we have
    $$ 4^\frac{n}{a} =  \left(4^{\frac{\log c}{\log 4}\frac{\log 4}{\log c}}\right)^\frac{n}{a} = \left(4^{\frac{\log c}{\log 4}}\right)^{\frac{n}{a}\frac{\log 4}{\log c}} = c^{\frac{n}{a}\frac{\log 4}{\log c}} \leq q^{\frac{n \log_c 4 }{a}}. $$
    We conclude that
    $$ \frac{q^n}{4^{\frac{n}{a}}}\geq \frac{q^n}{q^{\frac{n \log_c 4 }{a}}} = q^{n\left(1-\frac{\log_c 4}{a}\right)}= q^{n\frac{a-\log_c 4}{a}} $$
    Thus, our criterion can be reformulated as
    $$ q^{n} \geq \left(4\left(2+\frac{4r-1}{\delta'_\omega}\right) 2^{4(\omega-r) + b}\right)^\frac{2a}{a-\log_c 4} \quad \text{implies} \quad (q,n)\in N_4. $$
    From Lemma \ref{L: 4.3 in Lemos}, we may choose $(a,b,c) \in A$ where
    \begin{equation}\label{Set A}
        A=\{(1,0,41), (3,6,5) \} \cup \left\{\left(2, \frac{q-1}{2},q \right) \middle| 7 \leq q \leq 37  \right\}.
    \end{equation}
    Since $r$ may be chosen freely, we minimise the resulting lower bound on $q^n$ over $r$. Consequently, it suffices to consider
    \begin{equation*}
    M'_\omega = \min_{0 \leq r \leq \omega}\max_{(a,b,c) \in A} \left\{ \left(4\left(2 + \frac{4r-1}{\delta'_{\omega}} \right)2^{4(\omega-r)+b}\right)^{\frac{2a}{a-\log_c 4}} \right\}.
    \end{equation*}
    For $\omega = 852$, we obtain $M'_{852}<1.231 \times 10^{628}$. Thus, it suffices to show that $q^n \geq M_{852}$, since this implies $(q,n) \in N_4$ However, we also have that 
    $$q^n - 1 \geq \prod_{i=1}^{852}\hat{p}_i > 3.308 \times 10^{2821}, $$
    by virtue of $\omega = 852$. Thus, $\omega\geq 852$ implies $(q,n) \in N_4$. We proceed in a similar fashion for each $47\leq \omega\leq 851$ and conclude
    $$ \omega\geq 47 \quad \text{implies} \quad (q,n) \in N_4.$$
    Note that $M'_\omega$ is monotonically increasing. If $\omega\leq 46$, then 
    $$ q^n \geq M'_{46}\geq M'_{\omega} \quad \text{implies} \quad (q,n) \in N_4.$$
    Since $M'_{46} < 1.50 \times 10^{82}$, this completes the proof. 
    \end{proof}

\section{Varying $g$}
From now on, we may assume that any $(q,n) \notin N_4$ satisfies $q^{n} \leq 1.50 \times 10^{82}.$ Since $q\geq 5$, we have $n\leq \log_5(1.50 \times 10^{82})<118$. Therefore, we may also assume that $n\in [2,117]$ and $\omega \leq 46$. Consider the case $n \geq 24$. In this range,
$$ q \leq (M'_{46})^{1/24} < 2655, $$
and thus Criterion~\eqref{Sieve criterion} can be checked directly for each admissible prime power $q$ and exponent $n \in [24,117]$. Applying Proposition \ref{Sieving result} with $g=1,$ we find that all such pairs $(q,n)$ belong to $N_4,$ except possibly those displayed in Table ~\ref{Tab: Possible rejections, n geq 24}. 
\begin{table}[h!]
\centering
\begin{tabular}{|p{0.5cm}|p{2cm}|}
\hline
$q$ & \multicolumn{1}{c|}{$n$} \\
\hline
$5$  & $24,48,72,96$ \\
$7$  & $24,48,72,96$ \\
$11$ & $40,60,80$ \\
\hline
\end{tabular}
\caption{Possible rejections when $n\geq 24$}
\label{Tab: Possible rejections, n geq 24}
\end{table}

However, by choosing $g>1$, we are able to rule out the remaining exceptional pairs $(q,n)$. Thus, if $(q,n) \notin N_4$, then $n\leq 23$. We proceed by splitting the remaining cases 
$$R := \{(q,n) \notin N_4: q^n \leq 1.5 \times 10^{82}, \ 2 \leq n\leq 23\}$$
in terms of the size of $q$. For a given positive integer $Q_0$, we define the sets
$$ R^- := \{(q,n) \in R : q\leq Q_0\} \quad \text{and} \quad R^+:= \{(q,n) \in R : q> Q_0\}.$$
If we let $Q_0 = 10^4$ and consider $(q,n) \in R^+$, then 
$$n \leq \frac{\log(1.50 \times 10^{82})}{\log(q)} < \frac{\log(1.50 \times 10^{82})}{\log(Q_0)} < 21, $$
So $n\leq n_{\max} =20$. Applying Proposition~\ref{Sieving result} with $\delta_\omega$, $g=1$ and $s=n_{\max},$ we have
\begin{equation*}
    M_\omega = \min_{0 \leq r \leq \omega}\max_{(a,b,c) \in A} \left\{ \left(4\left(2 + \frac{4r+s-1}{\delta_{\omega}} \right)2^{4(\omega-r)+b}\right)^{\frac{2a}{a-\log_c 4}} \right\} .
    \end{equation*}
where $A$ is defined in \eqref{Set A}. Comparing $M_{\omega}$ with the trivial lower bound
$$ 1 + \prod_{i=1}^\omega \hat{p}_i, $$
we are able to rule out any pairs $(q,n)$ with $\omega>23$. Thus $(q,n) \in R^+$ implies $\omega \leq 23$. Minimising over all admissible choices of $r \in [1,23]$, in the sense that $\delta_{23}>0$, then $(q,n) \in R^+$ only if $q^n \leq M_{23}= 1.22 \times 10^{33}$. Finally, we may recompute
$$n \leq \frac{\log(1.22 \times 10^{33})}{\log(Q_0)} < 9. $$
Therefore, we have that
$$ R^+ = \{(q,n) \notin N_4 : q > 10^4, \ 2 \leq n \leq 8, \ q^n \leq 1.22 \times 10^{33}, \ \omega \leq 23\}. $$

Condition \eqref{Sieve criterion} can be directly checked for all $(q,n)$ such that both $q^n \leq 1.5 \times 10^{82}$ and $q \leq 10^4$ with $n \in [2,23]$. It was found that $\lvert R^-\rvert = 609$ and these possible exceptional pairs are included in $E_4$ in the appendix. It remains to determine $R^+.$ A computational check shows that there are no possible rejections for $n \geq 6.$ Hence
$$ R^+ = \{ (q,n) \notin N_4 : q>10^4, n \leq 5, q^n \leq 1.22 \times 10^{33}, \omega \leq 23 \}. $$

\section{Ramified Bounds}
Instead of bounding $q^n$ trivially from below via the $\omega$-th primorial, we use two results from \cite{bagger2026hybrid}. The improvement comes from factoring $x^n-1$ over $\Z[x]$ into cyclotomic polynomials $\Phi_d(x)$ and considering which residue classes a divisor $p \mid \Phi_d(q)$ must lie in. The exponent $n$ plays an important role, depending on whether it divides $q^n-1$. Let $\mathbb{P}$ denote the set of rational primes. For $p \in \mathbb{P},$ define 
$$\Pi_{p}^0:= \{p' \in \mathbb{P} : p' \not\equiv 1 \mod p \} \quad \text{and} \quad \Pi_{p}^1:= \mathbb{P}\setminus \Pi_p^0$$
with $\omega=\omega(q^p-1)$. For any $k \in [1,\omega]$, define $R_k^0:= \prod_{i=1}^{k}r_{i,0}$ and $R_{1}^{k}:= \prod_{j=1}^{\omega-k}r_{{j,1}}$ where $r_{i,s}\in \Pi_{p}^{s}$ is the $i^{th}$ element under the natural ordering.

\begin{theorem}[{\cite[Proposition $3$]{bagger2026hybrid}}]\label{p-ramification  bound}
Let $p \in \bb{P}$ with $x \in \bb{Z}_{>1}$ and assume, without loss of generality, that $p = r_{k'} \in \Pi_p^0$. Let $I_{k'}^- = [0,k')$ and $ I_{k'}^+= [k',\omega)$. Then we have
\begin{align*}
x \geq & \ \text{min}\left\{ \text{max}\left\{ 
1+R^0_k, \left(1 + R^0_k R^1_k\right)^{1/p} \right\} \ : \ k \in I_{k'}^- \right\} & \\
 &  \bigcup \left\{ \text{max}\left\{ 
1+p^{-1}R^0_{k+1}, \left(1 + p^{-1}R^0_{k+1}R^1_k\right)^{1/p} \right\} \ : \ k \in I_{k'}^+ \right\} & \ \ \text{for }  p \nmid x^p-1 \ \\
x \geq & \ \text{min} \left\{ \text{max}\left\{ 
1+p R^0_{k-1}, \left(1 + p^2 R^0_{k-1} R^1_k \right)^{1/p} \right\} \ : \ k \in I_{k'}^- \right\} & \\
&  \bigcup \left\{ \text{max}\left\{ 
1+R^0_k, \left(1 + p R^0_k R^1_k\right)^{1/p} \right\}  \ : \ k \in I_{k'}^+ \right\} & \ \ \text{for } p \mid x^p-1.  
\end{align*} 
\end{theorem}
This bound is used whenever the exponent $n$ is prime. Since a priori, we do not know if $p$ divides $x^p-1$, we take the minimum of the two bounds. When $n$ is a power of $2$, we also obtain the following improvement over the trivial lower bound.
\begin{corollary}\label{Hybrid bound for n=4}
Let $q$ be an odd prime power, $a \in \Z^+$ and $n=2^a$. Then
$$ q \geq \left(1+ 2n \prod_{i=1}^{\omega}p_i \right)^{\frac{1}{n}}.$$
\end{corollary}
This is~\cite[Theorem 3]{bagger2026hybrid}, specialised to the case where $q$ odd and $n=2$ or $4$. We will handle the remaining cases $n\leq 5$ individually and refer to the bounds given in Theorem~\ref{p-ramification  bound} and Corollary~\ref{Hybrid bound for n=4} collectively as \emph{ramified bounds}. Our method is stratified by $\omega=k$ and, whenever the search space becomes computationally manageable, we evaluate membership $(q,n) \in N_4$ directly. We proceed as follows.
 \begin{enumerate}
     \item Compare the sieving condition in Proposition~\ref{Sieving result} with the appropriate bound from Theorem~\ref{p-ramification  bound} or Corollary~\ref{Hybrid bound for n=4}.
     \item Employ a prime divisor tree for small primes dividing $q^n-1$.
     \item At each branch of the divisor tree, verify directly if any prime powers $q$ satisfy the divisibility conditions whilst $\omega = k$.
 \end{enumerate}
For each case, we may assume $\omega \leq 23$.
\subsection{Case n=5}\label{Case n=5}
Our search space is given by
$$10^{20} < q^5 \leq 1.22 \times 10^{33}, $$
and we fix $s=5$. Applying Theorem \ref{p-ramification  bound}, we obtain $q^5 \geq 4.122 \times 10^{38}$ when $\omega=23$. Since $M_{23}=7.189 \times 10^{32} < 4.122 \times 10^{38}$, we conclude that $(q,5) \in N_4$. The same comparison is carried out for the cases $\omega \in [17,22]$. The corresponding values of $M_\omega$ and the ramified lower bounds are given in Table \ref{M_omega for omega in [17,23] in n=5 case}.
\begin{table}[htbp]
\begin{tabular}{|
>{\centering\arraybackslash}m{1.5cm}|
>{\centering\arraybackslash}m{2cm}|
>{\centering\arraybackslash}m{3.5cm}|
>{\centering\arraybackslash}m{3.5cm}|}
    \hline
    $\omega$ & Optimal $r$ & $M_\omega $ when $s=5$ &  Ramified bound \\
    \hline
    $23$ &  $12$ & $7.189 \times 10^{32}$  & $4.122 \times 10^{38}$\\
    \hline
    $22$ &  $12$ & $1.654 \times 10^{31}$ & $3.431 \times 10^{36}$ \\
    \hline
    $21$ & $11$ & $3.928 \times 10^{30}$  & $5.868 \times 10^{34}$ \\
    \hline
    $20$ & $11$  & $3.065 \times 10^{29}$  & $2.354 \times 10^{32}$ \\
    \hline
    $19$ & $10$ & $2.292 \times 10^{28}$ & $9.377 \times 10^{29}$ \\
    \hline
     $18$ & $9$ & $6.264 \times 10^{27}$ & $3.751 \times 10^{28}$ \\
     \hline
    $\mathbf{17}$ & $\mathbf 9$ & $\mathbf{4.758 \times 10^{26}}$ & $\mathbf{1.706 \times 10^{26}}$ \\
      \hline
\end{tabular}
\caption{Comparison of $M_\omega$ and ramified lower bound for the cases $\omega \in [17,23]$.}
\label{M_omega for omega in [17,23] in n=5 case}
\end{table}
 
For each $\omega \geq 18,$ the ramified bound exceeds $M_\omega,$ and hence these cases are eliminated. It remains to consider the range $q^5 \leq M_{17}=4.758 \times 10^{26}.$ This range can be checked computationally. We find that every admissible $q$ in this range satisfies the sieve condition in Proposition \ref{Sieving result} and hence $(q,5) \in N_4$.
\subsection{Case $n=4$}
Our search space is given by
$$10^{16} < q^4 \leq 1.22 \times 10^{33}, $$
and we fix $s=4$. We apply Corollary \ref{Hybrid bound for n=4} and compare it with $M_\omega$ for $\omega \leq 23$. This allows us to eliminate the cases $\omega=22$ and $23$, as shown in Table \ref{Cases omega in [21,23] for n=4 case}.

\begin{table}[htbp]
\begin{tabular}{|
>{\centering\arraybackslash}m{1.5cm}|
>{\centering\arraybackslash}m{2cm}|
>{\centering\arraybackslash}m{3.5cm}|
>{\centering\arraybackslash}m{3.5cm}|}
    \hline
    $\omega$ & Optimal $r$ & $M_{\omega} $ when $s=4$ &  Ramified bound \\
    \hline
    $23$ &  $12$ & $6.906 \times 10^{32}$  & $2.136 \times 10^{33}$\\
    \hline
    $22$ &  $12$ & $1.586 \times 10^{31}$ & $2.574 \times 10^{31}$ \\
    \hline
    $\mathbf{21}$ & $\mathbf 11$ & $\mathbf{3.760 \times 10^{30}} $   & $\mathbf{3.258 \times 10^{29}}$ \\
    \hline
\end{tabular}
\caption{Comparison of $M_\omega$ and ramified lower bound for the cases $\omega \in [21,23].$}
\label{Cases omega in [21,23] for n=4 case}
\end{table}

We are left to determine whether $(q,4) \in N_4$ for $q^4 \leq M_{21}= 3.760 \times 10^{30}$ and $q>10^4,$ with $\omega \leq 21.$ 
We employ a variant of the prime divisor tree method used in
\cite{Trudgian and Jarso}. For each fixed $k$, we divide the problem into
subcases, according to whether selected small primes divide $q^4-1$. For a given (small) prime $p$, if $p$ divides $q^4-1$, then this imposes a restriction on the values $q^n-1$. If instead $p$ does not divide $q^4-1$, then $q^4-1$ must be large, making it more easily detectible by our sieve.

Consider the case $\omega=21$. Let $M_{\omega}^{a,b'}$ for some primes $a,b$ denote the value of  $M_\omega$ in the branch in which $a \mid q^4-1$  and $b \nmid q^4-1.$ Suppose first  that $3 \nmid q^4-1.$ In this case, we obtain  $M_{21}^{3'}=6.193 \times 10^{28}$ while the ramified bound gives $q^4 \geq 8.580 \times 10^{30}.$ Since, the latter exceeds $M_{21}^{3'}$, we can conclude that $(q,4) \in N_4$ throughout this case.

We next consider the case $3 \mid q^4-1.$ We further divide this branch according to whether $5$ divides $q^4-1.$ Suppose that $5 \nmid q^4-1$. Then $M_{21}^{3,5'}= M_{21}^{3'}$, since in the $\omega=21$ case, the optimal choice of $(x,r)$ minimising the quantity $4\Delta W(g)W(e)^4$ is $(10,11)$. Consequently, $M_{21}$ remains unchanged in all subcases in which one of the primes $\{p_1,\ldots, p_{10}\}$ does not divide $q^4-1$. We also obtain $q^4 \geq 5.148 \times 10^{30}$ using Corollary \ref{Hybrid bound for n=4}. Since $M_{21}^{3,5'}$ is smaller than this ramified lower bound, we again conclude that $(q,4) \in N_4$ throughout this subcase.

We are now left to consider the subcase in which both $3$ and $5$ divide $q^4-1.$ The process can be continued by further subdividing each branch according the divisibility of $q^4-1$ by subsequent primes. If we have made $l$ such divisibility conditions, involving the first $l$ primes other than $2,$ we say that we have descended to level $l$ of the prime divisor tree. We continue this process until reaching a subcase in which $M_{\omega}$ exceeds the ramified lower bound for $q^4$, summarised in Table \ref{Case n4 prime divisor tree method}.

\begin{table}[!htbp] 
\centering
\begin{tabular}
{|
>{\centering\arraybackslash}m{1cm}|
>{\centering\arraybackslash}m{1.5cm}|
>{\centering\arraybackslash}m{4cm}|
>{\centering\arraybackslash}m{4cm}|}
    \hline
    $\omega$ & Tree level & Interval of $q^4$ & \#$q$ satisfying conditions \\
    \hline
    $21$ & $11$  & $[3.258 \times 10^{29}, 3.760 \times 10^{30}]$ & $0$  \\ 
    \hline
    $20$ &  $7$ & $[4.463 \times 10^{27}, 2.915 \times 10^{29}]$& $0$  \\
    \hline
    $19$ &  $1$ & $[6.286 \times 10^{25}, 2.184 \times 10^{28}]$ & $0$ \\
    \hline
\end{tabular}
\caption{Prime divisor tree for $\omega \in [19,21]$ when $n=4$.}
\label{Case n4 prime divisor tree method}
\end{table}

Indeed, if $\omega \in [19,21],$ there are no values of $q$ for which $q^4-1$ simultaneously satisfies the divisibility conditions imposed by the prime divisor tree and has exactly $\omega$ distinct prime factors. Therefore, we conclude that $(q,4) \in N_4$ for all $q$ satisfying $q>10^4$ with $\omega \geq 19.$ 

It remains to consider $q$ satisfying  $q^4< M_{18}=5.949 \times 10^{27}, q>10^4$ with $\omega \leq 18$. For each fixed $\omega \leq 18$, we use SageMath to generate all prime powers $q$ with $q^4$ in the interval between the ramified lower bound and $M_\omega$, and determine which of these have exactly $\omega$ distinct prime divisors in $q^4-1.$
For each such $q$, we test whether the corresponding sieve condition in Proposition \ref{Sieving result} is satisfied. The results are summarised in Table \ref{Case n4 k leq 18}.
 
\begin{table}[htbp] 
\centering
\begin{tabular}
{|
>{\centering\arraybackslash}p{1cm}|
p{4cm}|
>{\centering\arraybackslash}p{2cm}|
>{\centering\arraybackslash}p{2cm}|}
    \hline
    $\omega$ & Interval of $q^4$  &  \#$q$ to check & \#rejections \\
    \hline
    $18$ &  $[9.383\times 10^{23}, 5.949 \times 10^{27}]$ & $0$  & $0$\\
    \hline
    $17$ &  $[1.538 \times 10^{22}, 4.495 \times 10^{26}]$ & $0$ & $0$ \\
    \hline
    $16$ & $[2.607 \times 10^{20}, 3.698 \times 10^{25}]$ & $0$   & $0$ \\
    \hline
    $15$ & $[4.919 \times 10^{18}, 9.480 \times 10^{24}]$  & $0$ & $0$  \\
    \hline
    $14$ & $[1.046 \times 10^{17}, 4.550\times 10^{23}]$  & $12$   & $0$  \\
    \hline
    $13$ & $[10^{16}, 4.226 \times 10^{22} ]$  & $45$  & $0$  \\
    \hline
    $12$ &  $[10^{16},1.103 \times 10^{22} ]$  & $239$ & $3$ \\
    \hline
    $11$ & $[10^{16}, 1.906 \times 10^{20}]$ & $320$ & $4$  \\
    \hline
    $10$ & $[10^{16}, 2.921 \times 10^{19}]$  &  $706$ &  $2$ \\
    \hline
    $9$  & $[10^{16}, 1.238 \times 10^{18}]$   & $572$  & $0$ \\
    \hline
    $8$ &  $[10^{16}, 6.973\times 10^{16}]$ & $240$ & $0$ \\
    \hline
\end{tabular}
\caption{Cases $\omega \leq 18$ when $n=4$ }
\label{Case n4 k leq 18}
\end{table}
For $\omega \leq 7,$ we have $M_\omega<10^{16}$. Since $q>10^4$, we have $q^4 > 10^{16},$ and hence there is nothing to check in these cases. Combining this with the results in Table~\ref{Case n4 k leq 18}, we conclude that $(q,4) \in N_4$ for $ 10^{16}<q^4 \leq 1.22 \times 10^{33}$ except possibly when 
$$q \in \{ 10009, 11131,13859,15313,17863,20747,21013,23561,53087.\}$$

\subsection{Case n=3}
Our search space is given by
$$10^{12} < q^3 \leq 1.22 \times 10^{33}, $$
and we fix $s=3$. As in the previous cases, we apply Proposition \ref{Sieving result} with $g=1$ to obtain values of $M_{\omega}$, and compare them with the lower bound from Theorem \ref{p-ramification  bound}. This allows us to reduce the bound on $\omega$ to $21$. By applying another prime divisor tree argument, we eliminate the cases $\omega=20,21$ as shown in Table \ref{Prime divisor tree for omega in [20,21] for n=3}.

\begin{table}[htbp] 
\centering
\begin{tabular}
{|
>{\centering\arraybackslash}p{1cm}|
>{\centering\arraybackslash}p{1.5cm}|
>{\centering\arraybackslash}p{4cm}|
>{\centering\arraybackslash}p{4cm}|}
    \hline
    $\omega$ &  Tree level & Interval of $q^3$ & \#$q$ satisfying conditions \\
    \hline
    $21$ & $12$  & $[4.343 \times 10^{29}, 3.596 \times 10^{30}]$   & $0$    \\ 
    \hline
    $20$ &  $11$ & $[9.242 \times 10^{27}, 2.769 \times 10^{29}]$ & $0$  \\
    \hline
\end{tabular}
\caption{Prime divisor tree for $\omega \in [20,21]$ when $n=3.$}
\label{Prime divisor tree for omega in [20,21] for n=3}
\end{table}

For $\omega\leq19$, the application of the prime divisor tree requires a more detailed analysis. As before, each branch is successively subdivided according to the divisibility of $q^3-1$ by small primes. However, unlike the cases $\omega=20,21$, the resulting branches cannot be eliminated solely by comparing the ramified bound with $M_\omega$. Consequently, the sieve condition must be verified computationally over the remaining interval for $q^3$ associated with each branch. At this stage, these intervals are still too large for a complete verification using SageMath. We therefore continue to refine the problem using the prime divisor tree until the remaining intervals are sufficiently small for such computations. We illustrate this procedure for the case $\omega=19$. We have that $M_{19}=2.078 \times 10^{28}$ while the ramified bound gives $q^3 \geq 8.973 \times 10^{25}.$ We now divide into two cases.
\begin{enumerate}
\item \textbf{Case 1:} $3 \nmid (q^3 - 1)$ - In this case, both $M_{19}$ and ramified bound improve, since the prime $3$ can be replaced by $p_{20}$ in the corresponding estimates. We obtain $M_{19}^{3'}= 1.082 \times 10^{27}$ and $q^3 \geq 1.026 \times 10^{27}$. Thus, the two bounds do not eliminate this case, and it remains to consider $q^3 \in [1.026 \times 10^{27}, 1.082 \times 10^{27}]$ subject to  $3 \nmid (q^3 - 1)$, and $\omega = 19$. A SageMath computation shows that $434456$ prime powers $q$ with $q^3$ in this interval satisfy $3\nmid(q^3-1)$, but none satisfies $\omega=19$. Hence, this case is eliminated.

\item \textbf{Case 2:} $3 \mid (q^3 - 1)$ - We further divide according to whether $5$ divides $q^3 - 1$.

\begin{enumerate}
\item \textbf{Case 2.1:} $5 \nmid (q^3 - 1)$ - In this subcase, the ramified bound improves to $q^3 \geq 1.848 \times 10^{27}$, while $M_{19}^{3,5'}=M_{19}^{3'} = 1.082 \times 10^{27}$. Since $M_{19}^{3,5'}$ is less than the lower bound for $q^3$, this subcase is eliminated.

\item \textbf{Case 2.2:} $5 \mid (q^3 - 1)$. 
\begin{enumerate}
    \item \textbf{Case 2.2.1:} $7 \nmid (q^3-1)$ - In this subcase, the ramified bound improves to $q^3 \geq 1.320 \times 10^{27}$, while $M_{19}^{3,5,7'}=1.082 \times 10^{27}$ remains unchanged. Thus, the two bounds are incompatible, and this branch is ruled out.

    \item \textbf{Case 2.2.2:} $7 \mid (q^3-1)$ - Proceeding to level $4$ of the prime divisor tree by considering the subcase $11\nmid(q^3-1)$ does not yield a sufficient improvement in the bounds to rule out this branch. Therefore, we stop at level $3$ and directly check the corresponding sieve condition in Proposition \ref{Sieving result} for $q$ such that $q^3 \in [8.973 \times 10^{25}, 2.078 \times 10^{28}]$ subject to $3,5,7 \mid q^3-1$ and $\omega=19.$ A SageMath computation shows that there are $6825191$ prime powers $q$ in this range satisfying $3,5,7 \mid q^3-1,$ but none satisfies $\omega=19.$ Hence, this branch is ruled out.
\end{enumerate} 

\end{enumerate}

\end{enumerate}
We have therefore established that $(q,3)\in N_4$ for
$q^3\leq 1.22\times10^{33}$ and $q>10^4$ in the case $\omega=19$.

We treat the remaining cases $\omega\leq18$ in a similar manner, with the results summarised below. A key difference in these cases is that the branch
where $3$ divides but $5$ doesn't divide $q^3-1$ can no longer be eliminated using the bounds alone. Indeed, for each $\omega \leq 18,$ the corresponding sieve bound $M_{\omega}^{3,5'}$ exceeds the lower bound for $q^3$ in this branch. Consequently, for each remaining value of $\omega$, we must perform
computational checks in the relevant ranges arising from three branches of the prime divisor tree.

\begin{table}[!ht] 
\vspace{3pt}
\centering
\begin{tabular}{|
>{\centering\arraybackslash}p{1cm}|
>{\centering\arraybackslash}p{5cm}|
>{\centering\arraybackslash}p{4cm}|}
\hline
$\omega$
&
$[\text{Ramified bound},\, M_{\omega}^{3'}]$
&
Possible rejections
\\
\hline
$18$ & $[9.970\times 10^{24},\, 8.118\times 10^{25}]$ & $0$ \\
\hline
$17$ & $[1.337\times 10^{23},\, 2.204\times 10^{25}]$ & $0$ \\
\hline
$16$ & $[2.506\times 10^{21}, 1.657 \times 10^{24}]$  &  $0$  \\
\hline
$15$ & $[3.173 \times 10^{19}, 1.361 \times 10^{23}]$   & $0$   \\
\hline
$14$ & $[1.094 \times 10^{18}, 3.467 \times 10^{22}]$  & $0$ \\
\hline
$13$ & $[1.498 \times 10^{16}, 1.659 \times 10^{21}]$  & $0$ - checks: 1   \\
\hline
\end{tabular}
\caption{Branch 1}
\label{Branch 1 of prime divisor tree for omega<=18 when n=3}
\end{table}

\begin{table}[!ht] 
\vspace{3pt}
\centering
\begin{tabular}{|
>{\centering\arraybackslash}p{1cm}|
>{\centering\arraybackslash}p{5cm}|
>{\centering\arraybackslash}p{4cm}|}
\hline
$\omega$
&
$[\text{Ramified bound},\, M_\omega^{3,5'}]$
&
Possible rejections
\\
\hline
$18$ & $[1.794\times 10^{25},\, 8.118\times 10^{25}]$ & $0$ \\
\hline
$17$ & $[1.850\times 10^{23},\, 2.204\times 10^{25}]$ & $0$ \\
\hline
$16$ & $[4.512 \times 10^{21}, 1.657 \times 10^{24}]$  & $0$   \\
\hline
$15$ & $[5.712 \times 10^{19}, 1.361 \times 10^{23}]$  & $0$  \\
\hline
$14$ & $[7.824 \times 10^{17}, 3.467 \times 10^{22}]$  & $0$  \\
\hline
$13$ & $[2.698 \times 10^{16}, 1.659 \times 10^{21}]$ & $0$ - checks: 1  \\
\hline
\end{tabular}
\caption{Branch 2}
\end{table}

\begin{table}[!ht]
\vspace{3pt}
\centering
\begin{tabular}{|
>{\centering\arraybackslash}p{1cm}|
>{\centering\arraybackslash}p{5cm}|
>{\centering\arraybackslash}p{4cm}|}
\hline
$\omega$
&
$[\text{Ramified bound},\, M_\omega^{3,5}]$
&
Possible rejections
\\
\hline
$18$ & $[2.188 \times 10^{24},\, 5.642\times 10^{27}]$ & $0$ \\
\hline
$17$ & $[2.256 \times 10^{22},\, 4.241\times 10^{26}]$ & $0$ \\
\hline
$16$ & $[2.856 \times 10^{20}, 3.484 \times 10^{25}]$  &  $0$\\
\hline
$15$ & $[9.848  \times 10^{18}, 8.875 \times 10^{24}]$  &  $0$- checks: 1 \\
\hline
$14$ & $[1.349 \times 10^{17}, 4.245 \times 10^{23}]$ & $0$ - checks: 2 \\
\hline
$13$ & $[2.148 \times 10^{15}, 3.916 \times 10^{22}]$ & $0$ - checks: 44 \\
\hline
\end{tabular}
\caption{Branch 3}
\end{table}

It remains to consider the region $q^3 \leq M_{12}=1.010 \times 10^{22}, q> 10^4$ with $\omega \leq 12.$ At this point, the remaining range is sufficiently small to perform an exhaustive computation in SageMath. This computation yields $96$ possible rejections in this region, which are included in the set $E_4$ in Appendix.

\subsection{Case n=2}
Our search space is given by
$$10^{8} < q^2 \leq 1.22 \times 10^{33}, $$
and we fix $s=3$. Comparing Proposition \ref{Sieving result} with the ramified bound for $q^2$ given by Corollary \ref{Hybrid bound for n=4}, we find that $M_{\omega}$ is smaller than the ramified bound only when $\omega=23.$ Hence, we are left to consider
$$q^2 \leq M_{22}= 1.454 \times 10^{31} \text{ and } q>10^4.$$
We attempted to further reduce this region using the prime divisor tree method together with the other techniques employed earlier in the article. However, the resulting computation was prohibitively expensive with the available computational resources. We therefore leave this remaining finite region unresolved. This completes the analysis of all cases of $n$. Combining the results obtained for all the cases of $n,$ we obtain Theorem \ref{T: main classification result}.

\section{Longer progressions} \label{Section on general m}
We turn our attention to progression of length $m>4$. For smaller values of $m$, we are still able to replicate some of the methods employed elsewhere in the paper to prove Theorem~\ref{T: medium m}. However, once the progressions become long enough, computational techniques become infeasible. There, we use Robin's bound on $\omega(n)$ to prove the general result Theorem~\ref{General m theorem (m>=11)}.
\subsection{Progressions with $5 \leq m \leq 10$}
 Let $P_r$ denote the product of the first $r$ primes. For each fixed $m$, our aim is to establish a result of the form
 \begin{equation*}
    q^n\geq P_r \quad \text{implies} \quad (q,n)\in N_m
\end{equation*}
We then apply this to $5 \leq m \leq 10$, reducing the problem to a finite region. Then using a similar technique employed in the treatment of $m=4$ case along with a computational search, we prove Theorem \ref{T: medium m}. 

\begin{proof}[Proof of Theorem \ref{T: medium m}]
Let $k>2m$ be a positive integer, to be chosen later, and let $r$ be the smallest positive integer satisfying
\begin{equation}
    \frac{1}{k} \geq \frac{r\log 2}{\log P_r}.
    \label{k condition}
\end{equation}
Then, by Lemma \ref{L: 4.1 in Lemos}, whenever $q^n \geq P_r$, 
\begin{equation}
    W(q^n-1) \leq q^{\frac{n}{k}}.
    \label{W qn bound}
\end{equation}
Suppose henceforth that $q^n \geq P_r$. Applying Proposition \ref{Sieving result} with $g=x^n-1$ and $e=q^n-1$, and using Lemma \ref{L: 4.3 in Lemos} and \eqref{W qn bound}, we have
\begin{equation*} 
mW(x^n-1)W(q^n-1)^m \leq m\,2^{\frac{n}{a}+b}q^{\frac{mn}{k}}.
\end{equation*}
Consequently, it is sufficient to verify
\begin{equation*}
    \frac{q^{\frac{n(k-2m)}{2k}}}{2^{\frac{n}{a}}}
    \geq m2^b.
\end{equation*}
We first consider sufficiently large values of $q$. Let $c$ be a positive integer, and suppose that $q \geq c= 2^{\log_2 c}.$ Then we have that $2^{\frac{n}{a}} \leq q^{\frac{n \log_c 2}{a}}.$ Therefore, it is sufficient to verify
\begin{equation}
    \frac{k-2m}{2k}-\frac{\log_c(2)}{a}>0 \quad \text{implies} \quad (q^n)^{ \frac{k-2m}{2k}-\frac{\log_c 2}{a} } \geq m2^b.
    \label{Sieve condition to check for general m}
\end{equation}
This condition is equivalent to $ c>2^{\frac{2k}{a(k-2m)}}.$
We begin with the choice $(a,b)=(1,0)$ from Lemma \ref{L: 4.3 in Lemos}. Let $c_1$ be an integer satisfying $c_1>2^{\frac{2k}{k-2m}}$ and take $c=c_1.$ For $q\geq c_1$, inequality \eqref{Sieve condition to check for general m} becomes
\begin{equation*}
    q^n \geq m^{\frac{2k}{(k-2m)-2k\log_{c_1}(2)}} =:M_1.
    \label{M1 condition}
\end{equation*}
Thus, if $P_r\geq M_1$, then $q^n\geq P_r$ implies $(q,n) \in N_m$ for every $q \geq c_1$.  In the computations, we seek the smallest admissible value of $c_1$ for which the resulting bounds can be satisfied. This is advantageous because it minimises the finite range of values of $q$ that must be considered separately.

It remains to consider the finite range $q<c_1$. We handle this by using the remaining choices of $(a,b)$ in Lemma \ref{L: 4.3 in Lemos}. In this finite range, we may take $c=q$ for each fixed $q$. Taking logarithms in \eqref{Sieve condition to check for general m} and rearranging  gives
\begin{align}
  n & \geq \frac{2ak(\log_2 m + b)}{a(k-2m)\log_2 q - 2k} =f_{a,b}(q), \nonumber
  \end{align}
provided $q>2^{\frac{2k}{a(k-2m)}}.$ Since $n$ is an integer, it is therefore sufficient to require $n\geq \left\lceil f_{a,b}(q)\right\rceil$. Thus, for any finite range of admissible values of $q$, we may define 
\begin{equation*}
    M=\max_q q^{\left\lceil f_{a,b}(q)\right\rceil},
    \label{general M condition}
\end{equation*}
where the maximum is taken over that range. If $P_r\geq M$, then $q^n\geq P_r$ implies $(q,n)\in N_m$ throughout the corresponding range of $q$.

We now illustrate this procedure with the remaining choices of $(a,b).$ For each such choice, choose a positive integer $c_a$ such that $c_a > 2^{\frac{2k}{a(k-2m)}}.$ Now, consider the range $c_a \leq q < c_{a-1}.$ For $a \geq 2$, we then define
\begin{equation*}
    M_a := \max_{c_a\leq q<c_{a-1}} q^{\left\lceil f_{a,b}(q)\right\rceil}.
\end{equation*}
If $P_r \geq M_{a},$ then $q^n \geq P_r$ implies $(q,n) \in N_m$ throughout the range $c_a \leq q<c_{a-1}.$ In our computations, the values $c_a$ are selected to be as small as possible, subject to the resulting bounds being attainable.

For a fixed $m$, we continue this process until the lower endpoint of the remaining range is at most $m$, since we are only interested in prime powers $q$ satisfying $\operatorname{char}(\mathbb{F}_q)\geq m$ and hence $q\geq m$. Thus, for a given $m$, it remains to choose $k$ and the corresponding $r$ satisfying \eqref{k condition} such that
\begin{equation}
    P_r\geq \max\{M_1,M_2,\ldots\}.
    \label{final Pr condition}
\end{equation}
A computer search determines suitable values of $k$ and $r$. Once \eqref{final Pr condition} is satisfied, we obtain that $q^n\geq P_r \quad \text{implies} \quad (q,n)\in N_m.$

\vspace{0.2cm}

The above procedure can be carried out for any fixed $m$ to find suitable choices of $k$ and $r$, reducing the problem to the finite region $q^n <P_r$ and hence to considering only $\omega\leq r.$ The remaining finite region can then be treated using methods similar to those employed earlier in the article.

Recall that, in the proof of Proposition \ref{Reduction of region to 10^82}, we applied Proposition \ref{Sieving result} with $g=x^n-1$ and considered each fixed value of $\omega$ separately. This allowed us to reduce the upper bound on $\omega$ from $852$ to $46$. More precisely, after restricting to $\omega \leq 46,$ we reduced the range that remains to be considered to $q^n < M_{46}.$ We apply the same procedure for $5\leq m\leq 10$, further reducing the range that remains to be considered and proving Theorem~\ref{T: medium m}. The resulting bounds are summarised in Table \ref{Results for general m}.

\begin{table}[!ht] 
\centering
\begin{tabular}{|
>{\centering\arraybackslash}m{1cm}|
>{\centering\arraybackslash}m{3cm}|
>{\centering\arraybackslash}m{3cm}|
>{\centering\arraybackslash}m{3cm}|}
    \hline
    $m$ & Choice of $(k,r)$  & $\omega(q^n-1)\leq u$ & $M_u$ \\\hline
    $4$ &  $(11,852)$   & $46$ & $1.50 \times 10^{82}$  \\ \hline
    $5$ &   $(13,2828)$    & $82$ & $8.591 \times 10^{171} $ \\ \hline
    $6$ & $(15,9632)$  &  $157$ & $1.269 \times 10^{384}$ \\ \hline
    $7$ & $(17,33501)$  & $319$  & $6.744 \times 10^{895}$ \\ \hline
    $8$ & $(19, 118465)$  &  $665$ & $1.420 \times 10^{2121}$\\ \hline
    $9$ & $(21, 424558)$  & $1424$ & $8.430 \times 10^{5086}$\\\hline
    $10$ & $(23, 664579)$ & $3090$ & $3.829 \times 10^{12228}$\\\hline
 \end{tabular}
 \caption{Bounds obtained for $4\leq m\leq 10$.}
 \label{Results for general m}
\end{table}
\end{proof}
The table also illustrates the increasing computational cost of this
approach. For example, when $m=10$, the bound on $\omega$ is
reduced from $664579$ to $3090$. However, obtaining this reduction
requires checking a large number of possible values of
$\omega$, and the number of cases that must be examined grows
rapidly with $m$. Thus, while the procedure gives substantially
improved bounds for each fixed $m$, carrying it out becomes
increasingly computationally demanding as $m$ grows.

 \subsection{Progressions with $m  \geq 11$}
Using Lemma~\ref{L: 4.1 in Lemos} to obtain an upper bound for $W(q^n-1)$, one can show the following.

\begin{proposition}[{\cite[Proposition 4.7]{lemos-neumann-ribas-2023}}]\label{Lemos bound for general m case}
Let $q$ be an odd prime power and $n,m$ be positive integers with $m \geq 5$ and $n \geq 2.$ Then
$$ q >   \exp{(17^{m+2})} \quad \text{implies} \quad (q,n) \in N_m. $$
\end{proposition}
We instead estimate this quantity using the following result of Robin~\cite{robin1983}.

  \begin{lemma} \label{Robin's bound}
  For $n  \geq 3,$
   \begin{equation*}
        \omega(n)  \leq  \frac{1.3841  \log n}{ \log \log n}.
   \end{equation*}
  \end{lemma}
\begin{proof}[Proof of Theorem~\ref{General m theorem (m>=11)}]
Let $m  \geq 11$ be an integer. Applying Proposition~\ref{Sieving result} together with $g=x^n-1$ and $e=q^n-1$, we see that $(q,n)  \in N_m$ provided
 \begin{equation*}
     q^{ \frac{n}{2}}  \geq mW(x^n-1)W(q^n-1)^m.
 \end{equation*}
 By Lemma  \ref{L: 4.3 in Lemos}, with the choice $(a,b)=(1,0),$ we have $W(x^n-1)  \leq 2^n$. On the other hand, Lemma  \ref{Robin's bound} gives
 $$W(q^n-1) = 2^{ \omega(q^n-1)}  \leq 2^{ \frac{1.3841  \log(q^n)}{ \log \log (q^n-1)}} = (q^n)^{ \frac{1.3841  \log 2}{ \log \log(q^n-1)}}.$$
Since $q \geq \operatorname{char} \mathbb{F}_q \geq m$ and $n \geq 2,$ we have that $q^n \geq m^2.$ Moreover, the function $\frac{\log\log x}{\log\log (x-1)}$ is decreasing for $x> e+1.$ Hence, 
\begin{equation*}
\frac{\log\log q^n}{\log\log (q^n-1)} \leq \frac{\log\log m^2}{\log \log (m^2-1)}=: A(m).
\end{equation*}

Therefore, it is sufficient to verify the condition
   \begin{align*}
        q^{n \left( 1- \frac{1.919mA(m) }{ \log \log (q^n)} \right)} &  \geq m^2 4^n .
   \end{align*}
   
Let $X=q^n$. Taking logarithms, we obtain
\begin{align}
    \left(1- \frac{1.919mA(m)}{\log\log X} \right)\log X & \geq 2 \log m + n \log 4. \nonumber \\
    \intertext{Since $n=\frac{\log X}{\log q} \leq \frac{\log X}{\log m}$, it is sufficient to check}
    \left( 1 - \frac{1.919mA(m)}{\log\log X}-\frac{\log 4}{\log m}\right)\log X & \geq 2 \log m. \label{Criterion to check for general m}
\end{align}
For fixed $m,$ the above inequality holds for sufficiently large $X$, provided the expression in the parentheses is positive. Choose $c \in \left(0,1-\frac{\log 4}{\log m}\right)$ and suppose that
$$ 1 - \frac{1.919mA(m)}{\log\log X}-\frac{\log 4}{\log m} \geq c  \quad \text{implying} \quad \frac{1.919mA(m)}{\left( 1-c-\frac{\log 4 }{\log m} \right)}  \leq\log \log X. $$ 
Therefore,
\begin{equation}
  X  \geq \exp\left( \exp \left( \frac{1.919mA(m)}{\left( 1-c-\frac{\log 4 }{\log m} \right)}  \right)  \right)  \label{Condition 1}.  
\end{equation}

Assuming \eqref{Condition 1}, the condition in \eqref{Criterion to check for general m} is satisfied whenever $X \geq m^{\frac{2}{c}}$. Then
\begin{equation}\label{Winning conditions}
 (q,n) \in N_m \quad \text{whenever} \quad X  \geq \max \left\{ m^{\frac{2}{c}} ,  \exp\left( \exp \left( \frac{1.919mA(m)}{\left( 1-c-\frac{\log 4 }{\log m} \right)}  \right)  \right) \right\}
\end{equation} 
for some $c \in \left(0, 1-\frac{\log 4}{\log m}\right). $ For convenience, we choose $c=0.001.$ We note that this choice is not optimal. With this choice of $c$, the condition in \eqref{Winning conditions} becomes
\begin{equation}\label{Winning condition with c=0.001}
    X \geq \max \left\{m^{2000}, \exp(\exp(mB(m))) \right\} \quad \text{where} \quad B(m):=\frac{1.919A(m)}{\left(0.999-\frac{\log 4}{\log m}\right)}. \end{equation}
We now determine when the second term in \eqref{Winning condition with c=0.001}  exceeds $m^{2000}$. This occurs when
\begin{align}
    \log 2000 + \log\log m \leq Bm. \label{Inequality to check for general result}
\end{align}

Define $F(x)=Bx- \log\log x - \log 2000.$ Then, $F'(x)= B - \frac{1}{x \log x}.$ Since we have $A(m)>1,$ then 
\begin{equation*}
    B(m) = \frac{1.919A(m)}{\left(0.999- \frac{\log 4}{\log m}  \right)} > \frac{1.919}{0.999} > 1.92 > \frac{1}{11 \log 11} \geq \frac{1}{x \log x},
\end{equation*}
for any $x \geq 11$. Therefore, $F'(x)>0$ for all $x \geq 11$, so $F$ is increasing on $[11, \infty). $ It suffices to verify that $F(11) \geq 0$. But $F(11) \geq 1.92(11) - \log\log 11 -\log 2000 >0$, so \eqref{Inequality to check for general result} holds and we conclude
\begin{equation*}
    X= q^n \geq \exp(\exp(mB(m)))  \quad \text{implies} \quad (q,n) \in N_m.
\end{equation*}
\end{proof}

\medskip

We now compare our bound in Theorem~\ref{General m theorem (m>=11)} with the one given in Proposition \ref{Lemos bound for general m case}. Their result states that 
\begin{align*}
    q &\geq \exp(17^{m+2}) = \exp \left(\exp (\log 17(m+2)) \right).
\end{align*}
Comparing the exponents in the two bounds, our bound is stronger whenever
\begin{equation}
mB(m) - \log 2 < \log 17(m+2) \quad \text{implying} \quad (B(m)-\log 17)m < \log(578). 
\label{Improvement over Lemos bound criterion check}
\end{equation}
For each $m,$ we compute $B(m)$ and certify the criterion above. One readily shows that the inequality holds for every integer $m \geq 53$. 

\section{Future work}
We conclude by outlining some future directions. In the case $m=4$, one should expect $E_4\cap N_4$ to be large. A direct verification of such pairs should be possible, at least when $q^n$ is small. Since the primitive elements in $\F_{q^n}$ are easily computable for sufficiently small fields, the remaining difficulty would be locating normal elements. Also, when $n=2$, one should be able to apply sieving methods to determine that almost all pairs $(q,2)$ with $q \leq 3.82\times 10^{15}$ lie in $N_4$. This would require a clever enumeration of all such $q$ and an efficient factorisation algorithm for $q^n-1$ to compute the sieve bound. It is also expected that one can augment the proof of Proposition~\ref{Sieving result} to obtain a stronger modified sieve as in~\cite[Theorem 1]{BaggerPunch2025}. This would give a slight improvement for the upper bounds on $q^n$ given by the sieve.

In the direction of longer progressions, one can employ a similar approach to $m=4$. After the reduction of $\omega$ by means of a $(e,g)=(q^n-1,x^n-1)$ type general sieve, one could make use of a composite variant of the ramified bound for exponents $n$ to reduce the region of unknown. For the exponents where the search-space is sufficiently small, the behaviour in the region could be computed directly. Similarly, branching techniques may be sufficient to eliminate certain large values for $n$.

\section*{Acknowledgements}
The authors thank Prof.~Tim Trudgian for suggesting the problem, for his helpful feedback, and for numerous insightful discussions.

\appendix
\section{Appendix}
We record the set $E_4$ of 714 potential exception pairs referenced in Theorem \ref{T: main classification result}.

{\scriptsize
\begin{flushleft}
$E_4=\{$(5, 3), (7, 3), (11, 3), (13, 3), (17, 3), (19, 3), (23, 3), (25, 3), (29, 3), (31, 3), (37, 3), (41, 3), (43, 3), (47, 3), (49, 3), (53, 3), (59, 3), (61, 3), (67, 3), (71, 3), (73, 3), (79, 3), (83, 3), (89, 3), (97, 3), (101, 3), (103, 3), (107, 3), (109, 3), (113, 3), (121, 3), (125, 3), (127, 3), (131, 3), (137, 3), (139, 3), (149, 3), (151, 3), (157, 3), (163, 3), (169, 3), (181, 3), (191, 3), (193, 3), (197, 3), (199, 3), (211, 3), (223, 3), (229, 3), (239, 3), (241, 3), (251, 3), (263, 3), (271, 3), (277, 3), (281, 3), (283, 3), (289, 3), (307, 3), (311, 3), (313, 3), (331, 3), (337, 3), (343, 3), (347, 3), (349, 3), (361, 3), (367, 3), (373, 3), (379, 3), (397, 3), (401, 3), (409, 3), (419, 3), (421, 3), (431, 3), (433, 3), (439, 3), (443, 3), (457, 3), (461, 3), (463, 3), (487, 3), (491, 3), (499, 3), (521, 3), (523, 3), (529, 3), (541, 3), (547, 3), (571, 3), (577, 3), (599, 3), (601, 3), (607, 3), (613, 3), (617, 3), (619, 3), (625, 3), (631, 3), (641, 3), (643, 3), (653, 3), (659, 3), (661, 3), (673, 3), (683, 3), (691, 3), (701, 3), (709, 3), (733, 3), (739, 3), (751, 3), (757, 3), (809, 3), (811, 3), (821, 3), (823, 3), (841, 3), (859, 3), (877, 3), (881, 3), (907, 3), (911, 3), (919, 3), (947, 3), (961, 3), (967, 3), (971, 3), (991, 3), (997, 3), (1009, 3), (1021, 3), (1031, 3), (1033, 3), (1051, 3), (1061, 3), (1087, 3), (1093, 3), (1117, 3), (1123, 3), (1129, 3), (1171, 3), (1201, 3), (1231, 3), (1291, 3), (1303, 3), (1321, 3), (1327, 3), (1331, 3), (1369, 3), (1381, 3), (1429, 3), (1451, 3), (1453, 3), (1471, 3), (1531, 3), (1597, 3), (1621, 3), (1681, 3), (1741, 3), (1759, 3), (1801, 3), (1831, 3), (1849, 3), (1861, 3), (1871, 3), (1873, 3), (1901, 3), (1933, 3), (1951, 3), (2011, 3), (2053, 3), (2081, 3), (2083, 3), (2131, 3), (2143, 3), (2161, 3), (2179, 3), (2209, 3), (2221, 3), (2251, 3), (2281, 3), (2311, 3), (2341, 3), (2347, 3), (2371, 3), (2377, 3), (2381, 3), (2401, 3), (2473, 3), (2521, 3), (2531, 3), (2551, 3), (2557, 3), (2591, 3), (2671, 3), (2713, 3), (2731, 3), (2791, 3), (2809, 3), (2851, 3), (2857, 3), (2971, 3), (3001, 3), (3019, 3), (3061, 3), (3121, 3), (3181, 3), (3259, 3), (3271, 3), (3301, 3), (3313, 3), (3331, 3), (3361, 3), (3389, 3), (3469, 3), (3481, 3), (3511, 3), (3541, 3), (3571, 3), (3631, 3), (3691, 3), (3697, 3), (3721, 3), (3851, 3), (3931, 3), (4027, 3), (4057, 3), (4111, 3), (4159, 3), (4201, 3), (4421, 3), (4447, 3), (4489, 3), (4561, 3), (4591, 3), (4603, 3), (4621, 3), (4651, 3), (4657, 3), (4663, 3), (4831, 3), (4951, 3), (5021, 3), (5041, 3), (5101, 3), (5209, 3), (5281, 3), (5413, 3), (5441, 3), (5521, 3), (5581, 3), (5749, 3), (5791, 3), (5821, 3), (5881, 3), (6007, 3), (6073, 3), (6091, 3), (6211, 3), (6217, 3), (6241, 3), (6271, 3), (6301, 3), (6421, 3), (6451, 3), (6491, 3), (6581, 3), (6661, 3), (6763, 3), (6841, 3), (6871, 3), (6997, 3), (7151, 3), (7177, 3), (7393, 3), (7459, 3), (7561, 3), (7723, 3), (7879, 3), (7921, 3), (8101, 3), (8191, 3), (8521, 3), (8581, 3), (8647, 3), (8779, 3), (8941, 3), (8971, 3), (9109, 3), (9181, 3), (9241, 3), (9283, 3), (9661, 3), (9811, 3), (10201, 3), (10609, 3), (10627, 3), (10651, 3), (10711, 3), (11131, 3), (11551, 3), (11701, 3), (11881, 3), (11971, 3), (12301, 3), (12541, 3), (13003, 3), (13399, 3), (13729, 3), (14071, 3), (14431, 3), (14641, 3), (14821, 3), (14851, 3), (15031, 3), (15271, 3), (15401, 3), (15661, 3), (16381, 3), (16831, 3), (16921, 3), (17161, 3), (17551, 3), (17761, 3), (17851, 3), (18481, 3), (19321, 3), (19471, 3), (19891, 3), (20101, 3), (21121, 3), (21211, 3), (22051, 3), (22441, 3), (22621, 3), (23011, 3), (23431, 3), (24061, 3), (24091, 3), (24781, 3), (25411, 3), (25561, 3), (25741, 3), (28051, 3), (28711, 3), (32041, 3), (32761, 3), (33331, 3), (33751, 3), (33931, 3), (34651, 3), (34981, 3), (36481, 3), (36871, 3), (37501, 3), (37951, 3), (38011, 3), (38431, 3), (39601, 3), (40231, 3), (41161, 3), (42331, 3), (42571, 3), (43891, 3), (44521, 3), (44851, 3), (46861, 3), (50311, 3), (51151, 3), (51613, 3), (51871, 3), (54121, 3), (55441, 3), (56611, 3), (59341, 3), (60589, 3), (61051, 3), (64171, 3), (81181, 3), (88741, 3), (91081, 3), (111211, 3), (115831, 3), (132661, 3), (133981, 3), (139129, 3), (162691, 3), (175561, 3), (197341, 3), (326041, 3), (5, 4), (7, 4), (11, 4), (13, 4), (17, 4), (19, 4), (23, 4), (25, 4), (29, 4), (31, 4), (37, 4), (41, 4), (43, 4), (47, 4), (49, 4), (53, 4), (59, 4), (61, 4), (67, 4), (71, 4), (73, 4), (79, 4), (83, 4), (89, 4), (97, 4), (101, 4), (103, 4), (107, 4), (109, 4), (113, 4), (121, 4), (125, 4), (127, 4), (131, 4), (137, 4), (139, 4), (149, 4), (151, 4), (157, 4), (163, 4), (167, 4), (169, 4), (173, 4), (179, 4), (181, 4), (191, 4), (193, 4), (197, 4), (199, 4), (211, 4), (223, 4), (227, 4), (229, 4), (233, 4), (239, 4), (241, 4), (251, 4), (257, 4), (263, 4), (269, 4), (271, 4), (277, 4), (281, 4), (283, 4), (289, 4), (293, 4), (307, 4), (311, 4), (313, 4), (317, 4), (331, 4), (337, 4), (343, 4), (347, 4), (349, 4), (353, 4), (359, 4), (361, 4), (367, 4), (373, 4), (379, 4), (389, 4), (397, 4), (401, 4), (409, 4), (419, 4), (421, 4), (431, 4), (433, 4), (439, 4), (443, 4), (449, 4), (457, 4), (461, 4), (463, 4), (467, 4), (487, 4), (491, 4), (499, 4), (503, 4), (509, 4), (521, 4), (523, 4), (529, 4), (541, 4), (547, 4), (557, 4), (563, 4), (569, 4), (571, 4), (577, 4), (587, 4), (593, 4), (599, 4), (601, 4), (613, 4), (617, 4), (619, 4), (631, 4), (643, 4), (647, 4), (659, 4), (673, 4), (683, 4), (701, 4), (727, 4), (733, 4), (743, 4), (757, 4), (761, 4), (769, 4), (797, 4), (811, 4), (827, 4), (829, 4), (839, 4), (853, 4), (857, 4), (859, 4), (881, 4), (883, 4), (911, 4), (919, 4), (937, 4), (953, 4), (967, 4), (1009, 4), (1013, 4), (1021, 4), (1033, 4), (1063, 4), (1091, 4), (1093, 4), (1123, 4), (1217, 4), (1231, 4), (1259, 4), (1289, 4), (1301, 4), (1303, 4), (1321, 4), (1331, 4), (1373, 4), (1399, 4), (1409, 4), (1427, 4), (1429, 4), (1483, 4), (1567, 4), (1583, 4), (1597, 4), (1607, 4), (1637, 4), (1693, 4), (1721, 4), (1747, 4), (1847, 4), (1849, 4), (1861, 4), (1871, 4), (1877, 4), (1931, 4), (1973, 4), (2003, 4), (2029, 4), (2087, 4), (2113, 4), (2129, 4), (2141, 4), (2143, 4), (2267, 4), (2309, 4), (2311, 4), (2393, 4), (2423, 4), (2437, 4), (2521, 4), (2549, 4), (2617, 4), (2729, 4), (2801, 4), (2843, 4), (2927, 4), (2939, 4), (2969, 4), (3011, 4), (3037, 4), (3079, 4), (3191, 4), (3359, 4), (3541, 4), (3557, 4), (3583, 4), (3613, 4), (3697, 4), (3719, 4), (3739, 4), (4003, 4), (4159, 4), (4217, 4), (4523, 4), (5237, 4), (5741, 4), (5851, 4), (6271, 4), (6469, 4), (6733, 4), (7853, 4), (8581, 4), (9043, 4), (9283, 4), (9547, 4), (10009, 4), (11131, 4), (13859, 4), (15313, 4), (17863, 4), (20747, 4), (21013, 4), (23561, 4), (53087, 4), (5, 5), (7, 5), (11, 5), (13, 5), (19, 5), (25, 5), (31, 5), (37, 5), (43, 5), (49, 5), (61, 5), (71, 5), (211, 5), (421, 5), (631, 5), (5, 6), (7, 6), (11, 6), (13, 6), (17, 6), (19, 6), (23, 6), (25, 6), (29, 6), (31, 6), (37, 6), (41, 6), (43, 6), (47, 6), (49, 6), (53, 6), (59, 6), (61, 6), (67, 6), (71, 6), (79, 6), (89, 6), (101, 6), (103, 6), (109, 6), (121, 6), (131, 6), (139, 6), (169, 6), (179, 6), (181, 6), (191, 6), (199, 6), (211, 6), (229, 6), (373, 6), (419, 6), (439, 6), (571, 6), (5, 7), (7, 7), (11, 7), (13, 7), (5, 8), (7, 8), (11, 8), (13, 8), (17, 8), (19, 8), (23, 8), (29, 8), (41, 8), (43, 8), (5, 9), (7, 9), (11, 9), (5, 10), (7, 10), (11, 10), (13, 10), (5, 12), (7, 12), (11, 12), (13, 12), (23, 12), (5, 16)$\}$

\end{flushleft}
}

\end{document}